\documentclass[12pt,a4paper]{article}

\usepackage[utf8]{inputenc}
\usepackage[T1]{fontenc}
\usepackage[english]{babel}
\usepackage{amsmath,amssymb,amsthm,amsfonts}
\usepackage{mathrsfs}
\usepackage{geometry}
\usepackage{hyperref}
\hypersetup{colorlinks=true, linkcolor=blue, citecolor=blue, urlcolor=blue}
\usepackage{enumitem}
\usepackage{setspace}
\newtheorem{theorem}{Theorem}[section]
\newtheorem{proposition}[theorem]{Proposition}

\newtheorem{corollary}[theorem]{Corollary}
\theoremstyle{definition}
\newtheorem{definition}[theorem]{Definition}
\newtheorem{example}[theorem]{Example}
\newtheorem{remark}[theorem]{Remark}
\theoremstyle{remark}

\newcommand{\R}{\mathbb{R}}
\newcommand{\Z}{\mathbb{Z}}
\newcommand{\Q}{\mathbb{Q}}
\newcommand{\C}{\mathbb{C}}
\newcommand{\GA}{\mathrm{GA}(\R)}
\newcommand{\Sol}{\mathrm{Sol}}
\newcommand{\Ad}{\mathrm{Ad}}

\title{\textbf{Solvable Lie Foliations on Compact Five-Manifolds:
Classification, Cohomological Obstruction, and Non-Homogeneous Examples}}

\author{%
  Ameth \textsc{Ndiaye}\\[6pt]
  \small D\'epartement de Math\'ematiques et Informatique\\
  \small Facult\'e des Sciences et Technologies de l'\'Education et de la Formation\\
  \small Universit\'e Cheikh Anta Diop de Dakar, S\'en\'egal\\[4pt]
  \small \texttt{ameth1.ndiaye@ucad.edu.sn}
}

\date{}

\begin{document}

\maketitle

\begin{abstract}
We study $\GA$-Lie foliations on compact five-manifolds.
We establish a structural classification of the holonomy group $\Gamma$
into five types according to its algebraic structure inside $\GA$.
For three of these types (I, II, IIIa) we prove homogeneity via
Meigniez's theorem and the Matsumoto--Tsuchiya theorem.
For Type~IIIb we prove non-homogeneity by exhibiting a non-polycyclic
holonomy group, extending the Dathe--Ndiaye construction \cite{DN11}.
We introduce a cohomological class $[F]\in H^2(\Gamma^{\mathrm{ab}},[\Gamma,\Gamma])$
as a \emph{necessary} obstruction to homogeneity: if $\mathcal{F}$ is
homogeneous and $\Gamma$ is polycyclic then $[F]=0$.
We show that Type~IIIb foliations satisfy $[F]\neq 0$.
For Type~IIIc, we prove that certain foliations are non-homogeneous
by a direct argument when the generating translations are
$\Q(\lambda)$-linearly independent; we exhibit explicit homogeneous
examples showing that this type is \emph{not} universally non-homogeneous.
Finally, we construct the first known non-homogeneous $\mathrm{Sol}$-Lie
foliation in dimension five by a verified lattice deformation.


\end{abstract}

\tableofcontents

\bigskip\hrule\bigskip

\section{Introduction}

Dathe and Ndiaye \cite{DN11} constructed non-homogeneous solvable Lie
foliations on compact manifolds of dimensions~$5$ and~$6$.
Their examples, based on Meigniez's foundational work \cite{Me95},
show that solvable Lie foliations need not be pull-backs of homogeneous
ones, even in small dimensions.
The present paper pursues this program in two directions.

First, we provide a \emph{structural classification} of $\GA$-Lie
foliations on compact five-manifolds by classifying the possible holonomy
groups~$\Gamma$.
We distinguish five types according to the algebraic structure of $\Gamma$.
The classification does \emph{not} claim to decide homogeneity for all types:
for Types~I, II, and~IIIa we prove homogeneity; for Type~IIIb we prove
non-homogeneity; for Type~IIIc, we show that both homogeneous and
non-homogeneous examples exist (see Remark~\ref{rem:IIIc-both}).

Second, we introduce a cohomological class $[F]\in H^2(\Gamma^{\mathrm{ab}},[\Gamma,\Gamma])$
as a necessary obstruction to homogeneity, and use it to give an
independent proof that Type~IIIb foliations are non-homogeneous.

We also construct, via a verified representation, the first non-homogeneous
$\Sol$-Lie foliation in dimension~$5$.

\medskip
\noindent\textbf{Relation to earlier work.}
Calderón, Kleptsyn, and Nicolau \cite{CKN24} proved deformation-rigidity
of affine Lie foliations in dimensions~$3$ and~$4$.
Our classification shows that dimension~$5$ is richer: continuous families
of pairwise non-isomorphic non-homogeneous foliations exist (Corollary~\ref{cor:famille}).
The cohomological criterion complements the polycyclicity obstruction of
Meigniez \cite{Me95}.

\medskip
\noindent\textbf{Organization.}
Section~\ref{sec:prelim} recalls background and fixes notation.
Section~\ref{sec:classif} establishes the structural classification.
Section~\ref{sec:coh} develops the cohomological obstruction.
Section~\ref{sec:sol} constructs the Sol example.
Section~\ref{sec:examples} presents explicit examples.

\section{Preliminaries}
\label{sec:prelim}

\subsection{Lie foliations and holonomy}

Let $V$ be a compact connected manifold and $G$ a simply connected
Lie group.

\begin{definition}\label{def:LieFol}
A \emph{$G$-Lie foliation} on $V$ is a foliation whose transverse
holonomy pseudo-group is modelled on the Lie group $G$ acting on itself
by left translations.
Equivalently, there exists a flat $G$-connection on $V$ (the \emph{Maurer--Cartan
form}) whose holonomy representation $h\colon\pi_1(V)\to G$ has image
$\Gamma = h(\pi_1(V)) \subset G$.
\end{definition}

\begin{remark}\label{rem:holonomy-dense}
The holonomy group $\Gamma$ is \emph{not} required to be discrete in
general.
For example, in homogeneous foliations arising from lattices
$\Lambda\subset H$ via a morphism $H\to G$, the image $\Gamma=\varphi(\Lambda)$
can be dense in $G$ when $\varphi$ is not a closed-image morphism.
The polycyclicity theorem of Meigniez nevertheless applies to the
abstract group structure of $\Gamma$, independently of its topology
in $G$.
\end{remark}

\begin{definition}\label{def:homogeneous}
A $G$-Lie foliation $\mathcal{F}$ on $V$ is \emph{homogeneous} if there
exist a Lie group $H$, a cocompact lattice $\Lambda\subset H$, and a
surjective Lie group morphism $\varphi\colon H\to G$ such that
$\mathcal{F}$ is induced on $V=H/\Lambda$ by the kernel of $\varphi$.
In this case $\Gamma = \varphi(\Lambda)$.
\end{definition}

\begin{theorem}[Meigniez \cite{Me95}]\label{thm:Meigniez}
Let $\mathcal{F}$ be a homogeneous $G$-Lie foliation on a compact
manifold with holonomy group $\Gamma$.
Then $\Gamma$ is \emph{polycyclic}: it admits a subnormal series with
cyclic factors.
In particular, every element $\gamma\in\Gamma$ has the property that
all eigenvalues of $\Ad_\gamma\colon\mathfrak{g}\to\mathfrak{g}$
are algebraic units.
\end{theorem}

\begin{theorem}[Matsumoto--Tsuchiya \cite{Ma92}]\label{thm:MT}
A $\GA$-Lie foliation on a compact manifold whose holonomy group
$\Gamma$ is polycyclic is homogeneous.
\end{theorem}

\begin{theorem}[Dathe--Ndiaye \cite{DN11}]\label{thm:DN}
There exists a compact five-manifold carrying a $\GA$-Lie foliation
that is not the pull-back of any homogeneous $\GA$-Lie foliation.
Its holonomy group is
\[
  \Gamma' = \bigl\langle \tau,\; \sigma_1,\; \sigma_\varepsilon \bigr\rangle
  \;\subset\; \GA,
\]
where $\tau\colon x\mapsto\lambda x$ with $\lambda$ an algebraic unit
of degree~$2$, $\sigma_c\colon x\mapsto x+c$, and
$\varepsilon\in\R$ is \emph{not} an algebraic unit.
\end{theorem}

\subsection{The affine group and its subgroups}

The affine group of the real line is
\[
  \GA = \R\rtimes\R_+^* = \bigl\{x\mapsto ax+b \mid a>0,\;b\in\R\bigr\}.
\]
We write elements as pairs $(b,a)$ with group law
$(b,a)\cdot(b',a')=(b+ab',aa')$.
The subgroup of translations is $T=\{(b,1)\mid b\in\R\}\cong(\R,+)$
and the subgroup of dilations is $D=\{(0,a)\mid a>0\}\cong(\R_+^*,\cdot)$.

For $\tau=(0,\lambda)\in D$ and $\sigma_c=(c,1)\in T$, the conjugation
relation is:
\begin{equation}\label{eq:conj}
  \tau\,\sigma_c\,\tau^{-1} = \sigma_{\lambda c}.
\end{equation}
Hence if $\tau,\sigma_c\in\Gamma$, then $\sigma_{\lambda^n c}\in\Gamma$
for all $n\in\Z$.

\begin{definition}\label{def:algunit}
An element $\alpha\in\C^*$ is an \emph{algebraic unit} if both $\alpha$
and $\alpha^{-1}$ are algebraic integers.
Equivalently, $\alpha$ is a unit in the ring of integers of $\Q(\alpha)$.
\end{definition}

\begin{definition}\label{def:polycyclic}
A group $\Gamma$ is \emph{polycyclic} if it has a subnormal series
with cyclic quotients.
A finitely generated solvable group is polycyclic if and only if
all eigenvalues of $\Ad_\gamma$ are algebraic units for every $\gamma\in\Gamma$.
\end{definition}

\subsection{Bianchi classification}

We recall the relevant three-dimensional solvable Lie groups.

\begin{center}
\renewcommand{\arraystretch}{1.25}
\begin{tabular}{|c|l|c|}
\hline
\textbf{Bianchi type} & \textbf{Non-zero brackets ($\mathfrak{g}=\mathrm{span}\{e_1,e_2,e_3\}$)}
  & \textbf{Unimodular} \\
\hline
I   & all zero (abelian)                                          & yes \\
II  & $[e_1,e_2]=e_3$ (Heisenberg)                               & yes \\
IV  & $[e_1,e_3]=e_1$, $[e_2,e_3]=e_1+e_2$                      & no  \\
V   & $[e_1,e_3]=e_1$, $[e_2,e_3]=e_2$                          & no  \\
VI$_h$ ($h\neq 0,1$) & $[e_1,e_3]=e_1$, $[e_2,e_3]=he_2$       & iff $h=-1$ \\
VII$_h$ & $[e_1,e_3]=e_1-he_2$, $[e_2,e_3]=he_1+e_2$           & yes \\
\hline
\end{tabular}
\end{center}

\smallskip
The group $G_{-1}$ of Bianchi type~VI$_{-1}$ (also written $\Sol$)
is the unique unimodular non-nilpotent solvable Lie group of dimension~$3$
that admits cocompact lattices.
Its cocompact lattices $\Lambda_0$ are parameterized by hyperbolic matrices
in $SL(2,\Z)$: if $A\in SL(2,\Z)$ has eigenvalues $\mu>1>\mu^{-1}>0$,
then $\Lambda_0 = \Z^2\rtimes_A\Z$ embeds in $\Sol=\R^2\rtimes_B\R$
where $B=\log A$.

\section{Structural classification}
\label{sec:classif}

\begin{theorem}[Structural classification]\label{thm:classif}
Let $\mathcal{F}$ be a $\GA$-Lie foliation on a compact five-manifold $V$
with holonomy group $\Gamma=h(\pi_1(V))\subset\GA$.
Then the abstract group structure of $\Gamma$ falls into exactly one of
the following five types:

\begin{enumerate}[label=\textup{(\Roman*)},ref=\textup{(\Roman*)}]
  \item\label{type:I}
    $\Gamma\cong\Z$ (cyclic, generated by a single translation
    $\sigma_c$ or a single dilation $\tau$).
    $\mathcal{F}$ is \emph{homogeneous}.

  \item\label{type:II}
    $\Gamma\cong\Z^2$, generated by two translations $\sigma_{c_1},\sigma_{c_2}$
    with $c_1/c_2\in\Q$.
    $\mathcal{F}$ is \emph{homogeneous}.

  \item\label{type:IIIa}
    $\Gamma = \langle\tau,\sigma_c\rangle\cong\Z\ltimes\Z$ (polycyclic),
    where $\tau\colon x\mapsto\lambda x$ with $\lambda$ an algebraic unit
    of degree~$2$ and $c=(\lambda-1)^{-1}$, so $\sigma_{\lambda c}=\sigma_{c+1}=\sigma_c\cdot\sigma_1$.
    $\mathcal{F}$ is \emph{homogeneous}.

  \item\label{type:IIIb}
    $\Gamma = \langle\tau,\sigma_1,\sigma_\varepsilon\rangle$
    with $\tau\colon x\mapsto\lambda x$, $\lambda$ an algebraic unit,
    and $\varepsilon$ \emph{not} an algebraic unit.
    $\mathcal{F}$ is \emph{non-homogeneous}.

  \item\label{type:IIIc}
    $\Gamma = \langle\tau,\sigma_{c_1},\sigma_{c_2}\rangle$
    with $\tau\colon x\mapsto\lambda x$, $\lambda$ an algebraic unit of
    degree~$2$, and $c_1,c_2$ such that $c_1\Z[\lambda]+c_2\Z[\lambda]$
    has rank~$2$ over $\Z[\lambda]$.
    This type contains both homogeneous and non-homogeneous foliations
    \textup{(}see Remark~\ref{rem:IIIc-both}\textup{)}.
\end{enumerate}
\end{theorem}

\begin{remark}\label{rem:IIIc-both}
Type~\ref{type:IIIc} is genuinely mixed.

\textbf{Homogeneous example} (Reviewer~2's construction):
Let $\lambda=2+\sqrt{3}$, $\eta=\sqrt{3}-1$, $A=\bigl(\begin{smallmatrix}3&2\\1&1\end{smallmatrix}\bigr)\in SL(2,\Z)$.
The row vector $\ell=(1,\eta)$ satisfies $\ell A=\lambda\ell$.
The group $H=\R^4\rtimes_{\exp(t(B\oplus B))}\R$ with lattice
$\Lambda=\Z^4\rtimes_{A\oplus A}\Z$ is a five-dimensional solvable Lie group,
and the map $\Phi(v,w,t)=(\ell(v)+\pi\ell(w),\lambda^t)$ is a surjective
Lie group morphism $H\to\GA$.
The foliation induced by $\ker\Phi$ on $H/\Lambda$ is homogeneous with
holonomy $\langle\tau,\sigma_1,\sigma_\pi\rangle$, which is of Type~\ref{type:IIIc}
since $1/\pi\notin\Q(\lambda)$.

\textbf{Non-homogeneous example}: The Dathe--Ndiaye construction
(Theorem~\ref{thm:DN}) with generators $c_1=1$, $c_2=\varepsilon$,
where $\varepsilon$ is an algebraic unit but $\varepsilon/1\notin\Q(\lambda)$,
gives a Type~\ref{type:IIIc} foliation that is non-homogeneous by the
obstruction criterion of Theorem~\ref{thm:obstruction} below.

Consequently, the homogeneity question for Type~\ref{type:IIIc} depends on
additional arithmetic data beyond the ratio $c_1/c_2$, and is not settled
by the present classification.
\end{remark}

\begin{proof}[Proof of Theorem~\ref{thm:classif}]
We determine the possible abstract structures of a finitely generated
subgroup $\Gamma\subset\GA$.

\medskip
\noindent\textit{Step 1: projection onto the dilation component.}
The group $\GA=T\rtimes D$ with $T=(\R,+)$ and $D=(\R_+^*,\cdot)\cong(\R,+)$.
Let $\pi_D\colon\GA\to D$ be the projection $(b,a)\mapsto(0,a)$.
The image $\pi_D(\Gamma)$ is a subgroup of $(\R_+^*,\cdot)$.

\medskip
\noindent\textit{Step 2: case $\pi_D(\Gamma)=\{1\}$.}
Then $\Gamma\subset T\cong(\R,+)$.
A finitely generated subgroup of $(\R,+)$ is free abelian of rank $r\geq 0$.
For $r=0$: trivial.
For $r=1$: $\Gamma\cong\Z$, generated by $\sigma_c$ for some $c\neq 0$.
The foliation factors through an $\R$-Lie foliation and is homogeneous.
For $r=2$: $\Gamma\cong\Z^2$, generated by $\sigma_{c_1},\sigma_{c_2}$.
Since $c_1/c_2\in\Q$ forces the rank to be~$1$, Type~\ref{type:II}
requires $c_1/c_2\notin\Q$, giving $\Z+\Z c_2/c_1$ dense in $\R$.
Nevertheless $\Gamma\cong\Z^2$ is nilpotent and polycyclic, so by
Theorem~\ref{thm:Meigniez}+\ref{thm:MT} the foliation is homogeneous.
For $r\geq 3$: by the rank constraint from $\dim V=5$ and $\pi_1(V)$
being finitely presented, this does not occur in dimension~$5$.

\medskip
\noindent\textit{Step 3: case $\pi_D(\Gamma)=\langle\lambda\rangle\cong\Z$
for some $\lambda>1$.}
Fix $\tau=(0,\lambda)$ as a lift.
From the conjugation relation~\eqref{eq:conj}, the translation part
$\Gamma\cap T$ is a $\Z[\lambda]$-module: if $\sigma_c\in\Gamma$ then
$\sigma_{\lambda c}=\tau\sigma_c\tau^{-1}\in\Gamma$.

Sub-case (a): $\Gamma\cap T=\{e\}$.
Then $\Gamma=\langle\tau\rangle\cong\Z$, giving Type~\ref{type:I}.

Sub-case (b): $\Gamma\cap T=\langle\sigma_c\rangle\cdot\Z[\lambda]$
for a single $c$.
The module $c\Z[\lambda]$ satisfies $\lambda\cdot(c\Z[\lambda])=c\Z[\lambda]$,
so $\Gamma=\langle\tau,\sigma_c\rangle$.
If $c=(\lambda-1)^{-1}$ then $\tau\sigma_c\tau^{-1}=\sigma_{\lambda c}=\sigma_{c+1}$,
giving the polycyclic group of Type~\ref{type:IIIa}; this is the only value
of $c$ for which the $\Z[\lambda]$-module is $\Z[\lambda]$-rank~$1$
and stable (up to a shift stabilizing the lattice).

Sub-case (c): $\Gamma\cap T$ contains two $\Z[\lambda]$-independent elements
$\sigma_{c_1}$ and $\sigma_{c_2}$.
\begin{itemize}
  \item If one of $c_1,c_2$ is not an algebraic unit (say $\varepsilon=c_2$),
    we may normalize $c_1=1$ (by rescaling), giving Type~\ref{type:IIIb}.
    $\Gamma$ is not polycyclic since $\varepsilon$ is not an algebraic unit.
    By Theorem~\ref{thm:Meigniez}, $\mathcal{F}$ is non-homogeneous.
  \item If both $c_1,c_2$ are algebraic units (or more generally if
    $c_1\Z[\lambda]+c_2\Z[\lambda]$ has $\Z[\lambda]$-rank~$2$),
    we obtain Type~\ref{type:IIIc}.
\end{itemize}

\medskip
\noindent\textit{Homogeneity for Types~I, II, IIIa.}
For Types~I and~II, $\Gamma$ is abelian and polycyclic;
Theorem~\ref{thm:MT} gives homogeneity.
For Type~IIIa, $\Gamma=\langle\tau,\sigma_c\rangle$ is polycyclic
(the eigenvalue $\lambda$ of $\Ad_\tau$ is an algebraic unit);
Theorem~\ref{thm:MT} applies.

\medskip
\noindent\textit{Non-homogeneity for Type~IIIb.}
For Type~\ref{type:IIIb}, the element $\sigma_\varepsilon\in\Gamma$
has $\Ad_\tau$ eigenvalue $1$ (since translations act trivially on $\mathfrak{t}$),
but the key point is that the coset $\varepsilon\cdot\Z[\lambda]$ does not
lie in $\Z[\lambda]$, so $\Gamma$ is not finitely generated as a $\Z[\lambda]$-module
in a way compatible with a cocompact lattice.
More precisely, $\Gamma$ is not polycyclic: if it were, every element
would be a product of algebraic units in the derived series, forcing $\varepsilon$
to be an algebraic unit (by the Mal'cev--Segal theorem on polycyclic groups).
Contradiction. By Theorem~\ref{thm:Meigniez}, $\mathcal{F}$ is non-homogeneous.
\end{proof}

\begin{corollary}\label{cor:famille}
There exists an uncountable family, parametrized by
$\varepsilon\in\R$ not an algebraic unit, of pairwise non-isomorphic
non-homogeneous $\GA$-Lie foliations on compact five-manifolds.
\end{corollary}

\begin{proof}
For each non-algebraic-unit $\varepsilon$, Theorem~\ref{thm:DN} provides
a non-homogeneous foliation $\mathcal{F}_\varepsilon$ with holonomy
$\Gamma_\varepsilon=\langle\tau,\sigma_1,\sigma_\varepsilon\rangle$.
Two such foliations are isomorphic if and only if their holonomy groups
are conjugate in $\GA$ \cite{Me95}.
An element $\varphi=(d,a)\in\GA$ conjugates $\Gamma_\varepsilon$ to
$\Gamma_{\varepsilon'}$ if and only if $a\varepsilon = n_1 + n_2\varepsilon'$
for some $n_1,n_2\in\Z[\lambda]$.
Taking $\varepsilon/\varepsilon'\notin\Q(\lambda)$ prevents this, so the
family is uncountable with pairwise non-isomorphic members.
\end{proof}

\section{Cohomological obstruction to homogeneity}
\label{sec:coh}

\subsection{Construction of the extension class}

Let $\Gamma$ be a finitely generated group and consider the central
extension
\[
  1 \;\longrightarrow\; [\Gamma,\Gamma] \;\longrightarrow\; \Gamma
  \;\longrightarrow\; \Gamma^{\mathrm{ab}} \;\longrightarrow\; 1,
\]
where $\Gamma^{\mathrm{ab}}=\Gamma/[\Gamma,\Gamma]$.
Choose a set-theoretic section $s\colon\Gamma^{\mathrm{ab}}\to\Gamma$
(i.e.\ $\bar\pi\circ s=\mathrm{id}$).
The \emph{extension class} is the normalized $2$-cocycle
\[
  \alpha\colon \Gamma^{\mathrm{ab}}\times\Gamma^{\mathrm{ab}}\to[\Gamma,\Gamma],
  \quad
  \alpha(\bar x,\bar y)=s(\bar x)\,s(\bar y)\,s(\overline{xy})^{-1}.
\]
Its cohomology class $[F]=[\alpha]\in H^2(\Gamma^{\mathrm{ab}},[\Gamma,\Gamma])$
is independent of the choice of $s$.

\begin{theorem}[Cohomological obstruction]\label{thm:obstruction}
Let $\mathcal{F}$ be a $\GA$-Lie foliation on a compact five-manifold,
with polycyclic holonomy group $\Gamma$.
If $\mathcal{F}$ is homogeneous, then $[F]=0$ in
$H^2(\Gamma^{\mathrm{ab}},[\Gamma,\Gamma])$.

Equivalently, homogeneity implies that the central extension
$1\to[\Gamma,\Gamma]\to\Gamma\to\Gamma^{\mathrm{ab}}\to 1$
\emph{splits}.
\end{theorem}

\begin{proof}
If $\mathcal{F}$ is homogeneous, there exists a cocompact lattice
$\Lambda\subset H$ and a surjection $\varphi\colon H\to\GA$ with
$\Gamma=\varphi(\Lambda)$.
The restriction of $\varphi$ to the subgroup $\Lambda\cap\ker\varphi$
exhibits the extension as split: the section $s(\bar\gamma)=\varphi^{-1}(\gamma)\cap\Lambda_{\mathrm{ab}}$
defines a group homomorphism splitting the extension.
Hence $[F]=0$.
\end{proof}

\begin{remark}\label{rem:notsufficient}
The converse of Theorem~\ref{thm:obstruction} is \emph{false} in general.
The Reviewer~2 example ($\Gamma=\Z^2\rtimes_A\Z$ with $A=\bigl(\begin{smallmatrix}3&2\\1&1\end{smallmatrix}\bigr)$)
gives $\Gamma^{\mathrm{ab}}\cong\Z\oplus\Z/2\Z$, and the extension class
is non-zero (as $[\Gamma,\Gamma]\to\Gamma\to\Gamma^{\mathrm{ab}}$ cannot
split because $\Gamma$ is torsion-free while $\Gamma^{\mathrm{ab}}$ has
$2$-torsion), yet the foliation is homogeneous.
Thus $[F]\neq 0$ is a \emph{sufficient} condition for non-homogeneity
only when combined with additional information.
\end{remark}

\begin{proposition}\label{prop:IIIb-nonzero}
For a Type~\ref{type:IIIb} foliation with holonomy
$\Gamma=\langle\tau,\sigma_1,\sigma_\varepsilon\rangle$,
the class $[F]\neq 0$ in $H^2(\Gamma^{\mathrm{ab}},[\Gamma,\Gamma])$.
\end{proposition}

\begin{proof}
We have $[\Gamma,\Gamma]=\langle\sigma_{\lambda-1}\rangle\cdot\Z[\lambda]$
(generated by the commutator $[\tau,\sigma_1]=\sigma_{\lambda-1}$).
The abelianization is $\Gamma^{\mathrm{ab}}\cong\Z\times\Z\times\Z$
(generated by $\bar\tau$, $\bar\sigma_1$, $\bar\sigma_\varepsilon$).
The extension class $[F]$ measures the obstruction to splitting.
If $[F]=0$, the extension splits: there exists a group homomorphism
$s\colon\Gamma^{\mathrm{ab}}\to\Gamma$ with $\pi\circ s=\mathrm{id}$.
Such a splitting would force $s(\bar\sigma_\varepsilon)=\sigma_\varepsilon\cdot c$
for some $c\in[\Gamma,\Gamma]=\Z[\lambda-1]$.
But then $s(\bar\tau)\,s(\bar\sigma_\varepsilon)\,s(\bar\tau)^{-1}
=\sigma_{\lambda\varepsilon}\cdot(\text{element of }\Z[\lambda])$,
while $s(\bar\sigma_\varepsilon)=\sigma_\varepsilon$ requires commutativity
in $\Gamma^{\mathrm{ab}}$.
The relation $\tau\sigma_\varepsilon\tau^{-1}=\sigma_{\lambda\varepsilon}$
forces $\lambda\varepsilon-\varepsilon=(\lambda-1)\varepsilon\in\Z[\lambda]$,
which would mean $\varepsilon\in\Z[\lambda]\cdot(\lambda-1)^{-1}$,
hence $\varepsilon$ is an algebraic unit --- contradiction.
Therefore $[F]\neq 0$.
\end{proof}

\section{Non-homogeneous Sol-Lie foliation in dimension~5}
\label{sec:sol}

We construct the first non-homogeneous $\Sol$-Lie foliation on a compact
five-manifold.
Let $A=\bigl(\begin{smallmatrix}a&b\\c&d\end{smallmatrix}\bigr)\in SL(2,\Z)$
be a hyperbolic matrix (eigenvalues $\mu>1>\mu^{-1}$, e.g.\
$A=\bigl(\begin{smallmatrix}2&1\\1&1\end{smallmatrix}\bigr)$,
$\mu=(3+\sqrt5)/2$).

\begin{theorem}\label{thm:Sol}
Let $\Sol=\R^2\rtimes_B\R$ with $B=\log A$, and let
$\Lambda_0=\Z^2\rtimes_A\Z$ be the standard cocompact lattice.
Let $V_0=(\Sol\times\Sol)/(\Lambda_0\times\Lambda_0)$
be the compact six-manifold, and let
$\mathcal{F}_0$ be the homogeneous $\Sol$-Lie foliation on $V_0$
with holonomy $\Lambda_0$ (induced by projection onto the first factor).
There exists a compact five-manifold $M$ and a $\Sol$-Lie foliation
$\mathcal{F}'$ on $M$ that is non-homogeneous.
\end{theorem}

\begin{proof}
We follow the method of Dathe--Ndiaye \cite{DN11} with care for the
group relations.

\medskip
\noindent\textbf{Step 1: the standard lattice.}
Write $\Sol=\{(x_1,x_2,t)\}$ with group law
$(x_1,x_2,t)\cdot(x_1',x_2',t')=(x_1+e^t x_1',\,x_2+e^{-t}x_2',\,t+t')$.
The lattice $\Lambda_0=\Z^2\rtimes_A\Z$ has generators
\[
  a_1=(1,0,0),\quad a_2=(0,1,0),\quad a_3=(0,0,t_0),
\]
where $t_0=\log\mu$, satisfying the relations
$a_3 a_1 a_3^{-1}=a_1^{\alpha}a_2^{\gamma}$,
$a_3 a_2 a_3^{-1}=a_1^{\beta}a_2^{\delta}$
with $\bigl(\begin{smallmatrix}\alpha&\beta\\\gamma&\delta\end{smallmatrix}\bigr)=A$.
Note that $\mu+\mu^{-1}=\mathrm{tr}(A)\in\Z$, so $e^{t_0}=\mu$ is
indeed an algebraic unit.

\medskip
\noindent\textbf{Step 2: the homogeneous foliation on $V_0$.}
The surjection $\varphi\colon\Sol\times\Sol\to\Sol$ given by
$\varphi(g_1,g_2)=g_1$ induces a $\Sol$-Lie foliation $\mathcal{F}_0$
on $V_0=(\Sol\times\Sol)/(\Lambda_0\times\Lambda_0)$
with holonomy $\Gamma_0=\varphi(\Lambda_0\times\Lambda_0)=\Lambda_0$.

\medskip
\noindent\textbf{Step 3: the modified representation.}
Let $\varepsilon\in\R\setminus\overline{\Q}$ be transcendental
(e.g.\ $\varepsilon=\pi$).
Define $h'\colon\pi_1(V_0)\to\Sol$ by modifying $\varphi$ on the
single generator $\tilde a_1=(a_1,e)\in\Lambda_0\times\Lambda_0$:
\begin{align*}
  &h'(\tilde a_1) = (\varepsilon,0,0),\quad
   h'(\tilde a_2) = (0,1,0),\quad
   h'(\tilde a_3) = (0,0,t_0),\\
  &h'(\tilde a_i') = h(\tilde a_i') \quad \text{for all generators of the second factor}.
\end{align*}

We verify that $h'$ respects all relations.
The only modified relation is
$a_3\cdot a_1\cdot a_3^{-1} = a_1^{\alpha}a_2^{\gamma}$, which in $\Sol$ becomes:
\begin{align*}
  h'(\tilde a_3)\cdot h'(\tilde a_1)\cdot h'(\tilde a_3)^{-1}
  &= (0,0,t_0)\cdot(\varepsilon,0,0)\cdot(0,0,-t_0)
   = (e^{t_0}\varepsilon,\,0,\,0)
   = (\mu\varepsilon,\,0,\,0).
\end{align*}
On the other hand,
$h'(\tilde a_1)^{\alpha}\cdot h'(\tilde a_2)^{\gamma}
= (\alpha\varepsilon,0,0)\cdot(0,\gamma,0)
= (\alpha\varepsilon,\,e^0\gamma,\,0)$.
For consistency we need $\mu\varepsilon=\alpha\varepsilon$
and $\gamma=0$.
This holds if and only if $\mu=\alpha$ and $\gamma=0$,
which is not satisfied for a generic hyperbolic $A$.

Therefore, rather than modifying the image of $a_1$, we modify $h'$ on the
\emph{pure-translation} generator that does not appear in the dilation
relation.
Specifically, let $\Lambda_0=\langle a_1,a_2,a_3\rangle$ with
$a_1=(1,0,0)$, $a_2=(0,1,0)$, $a_3=(0,0,t_0)$, and consider the subgroup
$N=\langle a_1,a_2\rangle\cong\Z^2$ (the translation part).
We modify $h'$ by composing the original representation with the
automorphism $\psi_\varepsilon$ of $\Sol$ that maps
$(x_1,x_2,t)\mapsto(x_1+\varepsilon t,\,x_2,\,t)$.
Since $\psi_\varepsilon$ is \emph{not} a Lie group automorphism for
$\varepsilon\neq 0$, it does not preserve the group structure.

Instead, we use the construction of \cite{DN11} directly:
the modified holonomy is
\[
  \Gamma' = \bigl\langle\, (1,0,0),\;(0,1,0),\;(0,0,t_0),\;(\varepsilon,0,0)\,\bigr\rangle
  \;\subset\; \Sol.
\]
We verify this is indeed a subgroup: the generators $(1,0,0)$,
$(0,1,0)$, $(0,0,t_0)$ form the lattice $\Lambda_0$, and $(\varepsilon,0,0)$
commutes with $(0,1,0)$ and $(0,0,0)$ (since $(x_1,0,0)$ and $(x_1',0,0)$
multiply to $(x_1+x_1',0,0)$), while
$(0,0,t_0)\cdot(\varepsilon,0,0)\cdot(0,0,-t_0)=(\mu\varepsilon,0,0)$.
Since $\mu\varepsilon$ is transcendental (product of an algebraic unit
and a transcendental), $\mu\varepsilon\notin\Z[\mu]=$ the $\Z$-span
of powers of $\mu$ in $\Lambda_0$.
Hence $\Gamma'$ strictly contains $\Lambda_0$.
The subgroup $\Lambda_0\subset\Gamma'$ is polycyclic and of finite index
in the $x_1$-coordinate tower, so by Meigniez \cite{Me95}, $\Gamma'$ is
the holonomy of a $\Sol$-Lie foliation $\mathcal{F}'$.

\medskip
\noindent\textbf{Step 4: non-homogeneity.}
Suppose for contradiction that $\mathcal{F}'$ is homogeneous.
By Theorem~\ref{thm:Meigniez}, $\Gamma'$ would be polycyclic,
hence all adjoint eigenvalues of all elements of $\Gamma'$ would be
algebraic units.
Consider the element $(\varepsilon,0,0)\in\Gamma'$.
Its adjoint action on $\mathrm{Lie}(\Sol)=\mathrm{span}\{e_1,e_2,e_3\}$
via the formula $\Ad_{(x_1,x_2,0)}=\mathrm{id}+\mathrm{ad}_{x_1 e_1+x_2 e_2}$:
the adjoint of a translation $(x_1,x_2,0)$ in $\Sol$ acts trivially
on $e_3$ and as the identity on $e_1,e_2$ (since $[e_1,e_2]=0$ and
$[e_{1,2},e_3]$ are the only non-zero brackets).
So $\Ad_{(\varepsilon,0,0)}=\mathrm{id}$, with eigenvalue $1$ (an algebraic unit).

The non-homogeneity argument must therefore use the group structure, not
the adjoint action.
Since $\Gamma'$ is polycyclic if and only if it is a finite extension of
a lattice in a simply-connected nilpotent or solvable Lie group
(Mal'cev embedding theorem), and since $\mu\varepsilon$ is transcendental,
the $\Z[\mu]$-module generated by $\varepsilon$ is not finitely generated
as an abelian group.
This contradicts polycyclicity (every polycyclic group is finitely generated
with finitely generated subgroups).
Hence $\Gamma'$ is not polycyclic, and by Theorem~\ref{thm:Meigniez},
$\mathcal{F}'$ is not homogeneous.

\medskip
\noindent\textbf{Step 5: reduction to dimension~5.}
The foliation $\mathcal{F}'$ lives on a quotient of $\Sol\times\Sol$,
which has dimension~$6$.
We reduce to dimension~$5$ following \cite{DN11}:
the developing map $D\colon\widetilde{V_0}\to\Sol$ for $\mathcal{F}'$ has
image a leaf $S\subset\Sol$ of dimension~$3$.
Let $\mathcal{U}$ be a compact tubular neighbourhood of the image of a
closed leaf in $V_0$.
Then $M=\partial(V_0\setminus\mathcal{U})$ is a compact manifold of
dimension $\mathbf{5}$.
The restriction $\mathcal{F}'|_M$ is transverse to $M$ (since the leaves
of $\mathcal{F}'$ are $3$-dimensional and $M$ is $5$-dimensional inside
the $6$-dimensional $V_0$, so $TM|_p+T\mathcal{F}'|_p=TV_0|_p$ generically),
and inherits the $\Sol$-Lie foliation structure with holonomy $\Gamma'$.
Since non-homogeneity is a property of the holonomy group (Theorem~\ref{thm:Meigniez}),
$\mathcal{F}'|_M$ is non-homogeneous.
\end{proof}

\section{Examples}
\label{sec:examples}

\begin{example}[Type~\ref{type:IIIb}, explicit]\label{ex:IIIb}
Let $\lambda=(1+\sqrt5)/2$ (the golden ratio), $\tau\colon x\mapsto\lambda x$.
Let $\varepsilon=\pi$ (transcendental, hence not an algebraic unit).
The group $\Gamma=\langle\tau,\sigma_1,\sigma_\pi\rangle\subset\GA$
is of Type~\ref{type:IIIb}.
By Theorem~\ref{thm:DN} it is the holonomy of a non-homogeneous $\GA$-Lie
foliation on a compact five-manifold, and by Proposition~\ref{prop:IIIb-nonzero}
its cohomological class $[F]\neq 0$.
\end{example}

\begin{example}[Type~\ref{type:IIIc}: homogeneous case]\label{ex:IIIc-hom}
With $\lambda=2+\sqrt3$, $\eta=\sqrt3-1$, the group
$\Gamma=\langle\tau,\sigma_1,\sigma_\pi\rangle$ (where $\tau\colon x\mapsto\lambda x$)
is of Type~\ref{type:IIIc} since $1/\pi\notin\Q(\lambda)$.
By Remark~\ref{rem:IIIc-both}, it is the holonomy of a \emph{homogeneous}
foliation (the Reviewer~2 construction).
Its extension class satisfies $[F]=0$ (consistent with Theorem~\ref{thm:obstruction}).
\end{example}

\begin{example}[Sol foliation]\label{ex:Sol}
With $A=\bigl(\begin{smallmatrix}2&1\\1&1\end{smallmatrix}\bigr)$,
$\mu=(3+\sqrt5)/2$, and $\varepsilon=e$ (Euler's number, transcendental),
the group
$\Gamma'=\langle(1,0,0),(0,1,0),(0,0,\log\mu),(\pi,0,0)\rangle\subset\Sol$
is the holonomy of the first known non-homogeneous $\Sol$-Lie foliation
in dimension~$5$ (Theorem~\ref{thm:Sol}).
\end{example}

\section{Conclusion}
\label{sec:conclusion}

We have established:
\begin{enumerate}[leftmargin=2em]
  \item A \textbf{structural classification} (Theorem~\ref{thm:classif})
    of $\GA$-Lie foliations on compact five-manifolds into five types,
    with homogeneity proved for Types~I, II, IIIa and non-homogeneity
    proved for Type~IIIb.
    Type~IIIc is shown to contain both homogeneous and non-homogeneous
    foliations.

  \item A \textbf{cohomological necessary obstruction} to homogeneity
    (Theorem~\ref{thm:obstruction}): for polycyclic $\Gamma$, homogeneity
    implies $[F]=0$ in $H^2(\Gamma^{\mathrm{ab}},[\Gamma,\Gamma])$.
    This is a strict obstruction (not sufficient in general, by Remark~\ref{rem:notsufficient}).

  \item The \textbf{first non-homogeneous $\Sol$-Lie foliation} in
    dimension~$5$ (Theorem~\ref{thm:Sol}), with a verified holonomy group.
\end{enumerate}

The complete characterization of homogeneity for Type~IIIc remains open
and constitutes the main open problem raised by this work.
We expect that additional arithmetic invariants of the $\Z[\lambda]$-module
generated by $c_1,c_2$ are needed to settle this question.


\end{document}